\pdfoutput=1
\documentclass{amsproc}
\usepackage{tikz}
\usepackage{amsfonts}
\usepackage[T1]{fontenc}
\usepackage{enumerate}
\newcommand{\KMS}{K_S(\mathbf{M})}
\newcommand{\KMSs}{K_S^*(\mathbf{M})}
\newcommand{\KKMS}{\mathbf{K}_S(\mathbf{M})}
\newcommand{\KKMSs}{\mathbf{K}_S^*(\mathbf{M})}
\theoremstyle{plain}

\newtheorem{corollary}{Corollary}

\newtheorem{definition}{Definition}

\newtheorem{lemma}{Lemma}

\newtheorem{proposition}{Proposition}
\newtheorem{remark}{Remark}

\newtheorem{theorem}{Theorem}
\theoremstyle{definition}
\newtheorem{example}{Example}
\newcommand{\NN}{\mathbb{N}}


\begin{document}
\title[Invariant means, complementary averages of means...]{ Invariant
means, complementary averages of means, and a characterization of the
beta-type means }
\author{Janusz Matkowski}
\address{Institute of Mathematics, University of Zielona G\'{o}ra, Szafrana
4a, PL-65-516 Zielona G\'{o}ra, Poland}
\email{j.matkowski@wmie.uz.zgora.pl}
\author{Pawe\l\ Pasteczka}
\address{Institute of Mathematics, Pedagogical University of Krakow, Podchor%
\k{a}\.{z}ych 2, PL-30-084 Krak\'{o}w, Poland}
\email{pawel.pasteczka@up.krakow.pl}

\begin{abstract}
We prove that whenever the selfmapping $(M_1,\dots,M_p)\colon I^p \to I^p$, (%
$p \in \mathbb{N}$ and $M_i$-s are $p$-variable means on the interval $I$)
is invariant with respect to some continuous and strictly monotone mean $K
\colon I^p \to I$ then for every nonempty subset $S \subseteq\{1,\dots,p\}$
there exists a uniquely determined mean $K_S \colon I^p \to I$ such that the
mean-type mapping $(N_1,\dots,N_p) \colon I^p \to I^p$ is $K$-invariant,
where $N_i:=K_S$ for $i \in S$ and $N_i:=M_i$ otherwise. Moreover 
\begin{equation*}
\min(M_i\colon i \in S)\le K_S\le \max(M_i\colon i \in S).
\end{equation*}

Later we use this result to: (1) construct a broad family of $K$-invariant
mean-type mappings, (2) solve functional equations of invariant-type, and
(3) characterize Beta-type means.
\end{abstract}

\maketitle

\section{Introduction}

We show that every symmetric and increasing $p$-variable mean $K$, that is
invariant with respect to a given mean-type mapping $\mathbf{M}=\left(
M_{1},...,M_{p}\right) $, generate a unique finite family of mean-type
mappings, the coordinates of which are referred to as the complementary $K$%
-averages of a respective subfamily of the coordinates means of $\mathbf{M}$%
. The basic symmetric mean $K$ remains invariant with respect to each member
of this family, which, in general does not include any iterates of $\mathbf{M%
}$. However, according to the invariance principle (\cite{JM2013}, see also, 
\cite{JM2009} and \cite{M-P}) the iterates of the each mean-type mapping
from this family, converge on compact subsets, to the mean-type map $\mathbf{%
K}=\left( K,...,K\right) $, that is important in effective solving some
functional equations (section~\ref{sec:ASFE}).

In a recent paper \cite{HimMat2018}, a purely structural property of Beta
Euler function, gave rise to introduce a family of new means $\mathcal{B}%
_{p} $, called beta-type means. \ In this note (section \ref{sec:Beta}) we
observe that the beta-type mean can be characterized via invariance identity
involving the classical geometric and arithmetic mean. 

\section{Preliminaries}

In the whole paper $I\subset \mathbb{R}$ stands for an iterval, $p\in 
\mathbb{N},$ $p>1$ is fixed, and $\mathbb{N}_{p}:=\{1,\dots ,p\}$.

A function $M:I^{p}\rightarrow I$ is called a \textit{mean} in $I$ if%
\begin{equation*}
\min \left( x_{1},...,x_{p}\right) \leq M\left( x_{1},...,x_{p}\right) \leq
\max \left( x_{1},...,x_{p}\right) \text{, \ \ \ \ \ }x_{1},...,x_{p}\in I\,%
\text{, }
\end{equation*}%
or, briefly, if%
\begin{equation*}
\min \mathbf{x}\leq M\left( \mathbf{x}\right) \leq \max \mathbf{x}\text{, \
\ \ \ \ }\mathbf{x=}\left( x_{1},...,x_{p}\right) \in I^{p}\text{.}\,\text{ }
\end{equation*}%
The mean $M$ is called \textit{strict}, if for all nonconstant vectors $%
\mathbf{x}$ these inequalities are sharp; and \textit{symmetric}, if $%
M\left( x_{\sigma \left( 1\right) },...,x_{\sigma \left( p\right) }\right)
=M\left( x_{1},...,x_{p}\right) $ for all $x_{1},...,x_{p}\in I$ and all
permutations $\sigma $ of the set $\mathbb{N}_{p}$. Mean $M$ is \emph{%
monotone} if it is increasing in each of its variables. It is important to
emphasize that every strictly monotone mean is strict.

A mapping $\mathbf{M}\colon I^{p}\rightarrow I^{p}$ is referred to as \emph{%
mean-type} if there exists some means $M_{i}\colon I^{p}\rightarrow I$, $%
i\in \mathbb{N}_p$, such that $\mathbf{M}=(M_{1},\dots ,M_{p})$. For a
mean-type mapping $\mathbf{M} \colon I^p \to I^p$ we denote a projection
onto the $i$-th coordinate by $[\mathbf{M}]_i \colon I^p \to I$. In this
case we obviously have $[\mathbf{M}]_i=M_i$ for all $i \in \mathbb{N}_p$.

We say that a function $K\colon I^{p}\rightarrow \mathbb{R}$ is invariant
with respect to $\mathbf{M}$ (briefly $\mathbf{M}$-invariant), if $K\circ 
\mathbf{M}=K$.

\begin{theorem}[Invariance Principle]
If $\mathbf{M}\colon I^{p}\rightarrow I^{p}$, $\mathbf{M}=(M_{1},\dots
,M_{p})$ is a continuous mean-type mapping such that 
\begin{equation*}
\max \mathbf{M}\left( \mathbf{x}\right) -\min \mathbf{M}\left( \mathbf{x}%
\right) <\max \left( \mathbf{x}\right) -\min \left( \mathbf{x}\right) \text{%
, \ \ \ }\mathbf{x\in }I^{p}\backslash \Delta \left( I^{p}\right) ,
\end{equation*}%
where $\Delta \left( I^{p}\right) :=\left\{ \text{\ }\mathbf{x=}\left(
x_{1},...,x_{p}\right) \in I^{p}:x_{1}=...=x_{p}\right\} $ then there is a
unique $\mathbf{M}$-invariant mean $K:I^{p}\rightarrow I$ and the sequence
of iterates $\left( \mathbf{M}^{n}\right) _{n\in \mathbb{N}}$ of the
mean-type mapping $\mathbf{M}$ converges to $\mathbf{K}:=\left( K,\dots
,K\right) $ pointwise on $I^{p}$.
\end{theorem}

\section{A family of complementary means}

\bigskip Recall the following

\begin{remark}
(\cite{JM1999})Assume that $K:I^{2}\rightarrow I$ is a symmetric mean which
is continuous and monotone.

Then

(i) for an arbitrary mean $M_{1}:I^{2}\rightarrow I$ there is a unique mean $%
M_{2}:I^{2}\rightarrow I$ such that $K$ is $(M_{1},M_{2})$-invariant, 
\begin{equation*}
K\circ (M_{1},M_{2})=K\text{;}
\end{equation*}%
$\mathcal{K}_{M_{1}}:=M_{2},$ is referred to as a $K$-\textit{complementary
mean for} $M_{1};$ and we have 
\begin{equation*}
\mathcal{K}_{M_{1}}=M_{2}\Longleftrightarrow \mathcal{K}_{M_{2}}=M_{1};
\end{equation*}%
(ii) \ if $M_{1},M_{2}:I^{2}\rightarrow I$ are means such that $K$ is $%
(M_{1},M_{2})$-invariant, then there exists a unique mean $%
M:I^{2}\rightarrow I$ such that 
\begin{equation*}
\min \left( M_{1},M_{2}\right) \leq M\leq \max \left( M_{1},M_{2}\right)
\end{equation*}%
and%
\begin{equation*}
K\circ (M,M)=K\text{;}
\end{equation*}%
moreover 
\begin{equation*}
M=K.
\end{equation*}
\end{remark}

%

In this case $p\geq 3$ the counterpart of part (i) of Remark 1 is false
which shows the following

\begin{example}
Let $I=\mathbb{R}$, and $K=A$ where $A\left( x_{1},x_{2},x_{3}\right) =\frac{%
x_{1}+x_{2}+x_{3}}{3}$. The functions $M_{1}\left( x_{1},x_{2},x_{3}\right) =%
\frac{x_{1}+x_{2}}{2},$ $M_{2}\left( x_{1},x_{2},x_{3}\right) =x_{2}$ are
means in $\mathbb{R}$. But it is easy to see that there is no mean $M_{3}$
such that 
\begin{equation*}
A\circ \left( M_{1},M_{2},M_{3}\right) =A\text{,}
\end{equation*}%
but a partial counterpart of part (ii) holds true.
\end{example}

\bigskip

\begin{figure}[!ht]
\center
\tikzstyle{universal}=[circle,draw=blue!50,fill=blue!20,thick,
inner sep=0pt,minimum size=6mm]
\tikzstyle{Sdependence}=[rectangle,draw=black!50,fill=black!5,thick, inner sep=2pt,minimum size=4mm] 
\tikzstyle{comment}=[draw=black,opacity=1,fill=white,thick,style=opaque,inner sep=2pt,minimum size=4mm]
\begin{tikzpicture}[auto]
\node at (0,4) [universal] (M) {$\mathbf{M}$};
\node at (2,2) [universal] (K) {$K$};
\node at (-2,2) [Sdependence] (MS*) {$\{M_i\}_{i \in \NN_p\setminus S}$};
\node at (0,0) [Sdependence] (KS) {$\KMS$};
\node at (2,-2) [Sdependence] (KS*) {$\KMSs$};
\draw [thick, dotted, ->] (M) -- (K);
\draw [thick, ->] (M) -- (MS*);
\draw [thick, ->] (MS*) to [bend left] (KS);
\draw [thick, ->] (K) to [bend right]  (KS);
\draw [thick, ->] (K) -- (KS*);
\draw [thick, ->] (KS)  to [bend left] (KS*);
\node at (1,3) [comment] {$\mathbf{M}$-invariance};
\node at (0,0.8) [comment] {$\KKMS$-invariance};
\node at (1.2,-1) [comment] {$\KKMSs$-invariance};
\end{tikzpicture}
\caption{\label{fig} Map of dependencies between means. Rectangle vertexes are dependent on $S$. Dotted line means that there could exists more $\mathbf{M}$-invariant mean satisfying the conditions of Theorem~\ref{thm:KS}.}
\end{figure}
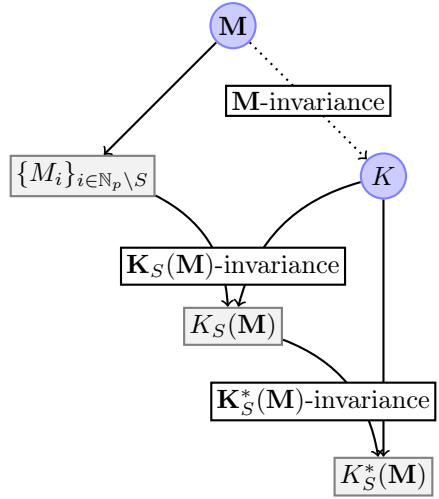

\bigskip

The main result of this section reads as follows.

\begin{theorem}
\label{thm:KS} Let $M_{1},\dots ,M_{p}\colon I^{p}\rightarrow I$ of means.
Assume that $K\colon I^{p}\rightarrow I$ is a continuous and monotone mean
which is invariant with respect to the mean type mapping $\mathbf{M}%
:=(M_{1},\dots ,M_{p})$.

Then for every nonempty subset $S\subseteq \mathbb{N}_{p}$ there exists a
unique mean $K_{S}(\mathbf{M})\colon I^{p}\rightarrow I$ such that $K$ in $%
\mathbf{K}_{S}(\mathbf{M})$-invariant, where $\mathbf{K}_{S}(\mathbf{M}%
)\colon I^{p}\rightarrow I^{p}$ is given by 
\begin{equation}
\lbrack \mathbf{K}_{S}(\mathbf{M})]_{i}:=%
\begin{cases}
K_{S}(\mathbf{M}) & \text{ for }i\in S, \\ 
M_{i} & \text{ for }i\in \mathbb{N}_{p}\setminus S.%
\end{cases}
\tag{1}
\end{equation}%
Moreover $\min (M_{i}\colon i\in S)\leq K_{S}(\mathbf{M})\leq \max
(M_{i}\colon i\in S)$.
\end{theorem}

\begin{proof}
In the case $S=\mathbb{N}_{p}$ the $\mathbf{K}_S(\mathbf{M})$-invariance of $%
K$ implies $K_S(\mathbf{M})=K$ and the statement is obvious. From now on we
assume that $S\neq \mathbb{N}_{p}$.

Denote briefly $M_{\vee }:=\max \{M_{i}:i\in S\}$ and $M_{\wedge }:=\min
\{M_{i}:i\in S\}$. Fix $\mathbf{x}\in I^{n}$ arbitrarily. Define a function $%
T\colon I\rightarrow I^{p}$ by 
\begin{equation*}
\big[T(\alpha )\big]_{i}:=%
\begin{cases}
M_{i}(\mathbf{x}) & \qquad i\in \mathbb{N}_{p}\setminus S, \\ 
\alpha & \qquad i\in S.%
\end{cases}%
\end{equation*}%
and $f\colon I\rightarrow I$ by $f(\alpha ):=K\circ T(\alpha )$. Then, as $K$
is continuous and strictly increasing, so is $f$. Therefore in view of the
equality $K\circ \mathbf{M}(\mathbf{x})=K(\mathbf{x})$ we obtain $f(M_{\vee
}(\mathbf{x}))\geq K(\mathbf{x})$ and $f(M_{\wedge }(\mathbf{x}))\leq K(%
\mathbf{x})$. Thus there exists unique number $\alpha _{\mathbf{x}}\in
\lbrack M_{\wedge }(\mathbf{x}),M_{\vee }(\mathbf{x})]$ such that $f(\alpha
_{\mathbf{x}})=K(\mathbf{x})$. Now, as $\mathbf{x}\in I^{n}$ was arbitrary
we define $K_S(\mathbf{M})(\mathbf{x}):=\alpha _{\mathbf{x}}$.

Then we have 
\begin{equation*}
K(\mathbf{x})=f(\alpha_{\mathbf{x}})=f(K_S(\mathbf{M})(\mathbf{x}))=\big(%
K\circ T\circ \big(K_S(\mathbf{M})\big)\big)(\mathbf{x})=\big(K\circ \big(%
\mathbf{K}_S(\mathbf{M})\big)\big)(\mathbf{x}),
\end{equation*}%
which shows that $K$ is $\mathbf{K}_S(\mathbf{M})$-invariant.

Now we need to show that $K_S(\mathbf{M})$ is uniquely determined. Assume
that $K$ is $\mathbf{K}_S(\mathbf{M})$-invariant and $K_S(\mathbf{M})(%
\mathbf{x})\neq \alpha _{\mathbf{x}}$ for some $\mathbf{x}\in I^{p}$. Then,
as $f$ is a monomorphism we obtain 
\begin{equation*}
K\circ \big(\mathbf{K}_S(\mathbf{M})\big)(\mathbf{x})=\big(K\circ T\circ %
\big(K_S(\mathbf{M})\big)\big)(\mathbf{x})\neq f(\alpha _{\mathbf{x}})=K(%
\mathbf{x})
\end{equation*}%
contradicting the $\mathbf{K}_S(\mathbf{M})$-invariance.
\end{proof}

%
%
%


The intuition beyond this theorem is the following. Once we have a
continuous and monotone mean $K$ such that $\mathbf{M}$ is $K$-invariant
mean we can unite a subfamily $(M_{s})_{s\in S}$ into a single mean (denoted
by $K_{S}(\mathbf{M})$) to preserve the $K$-invariance. 
In view of Theorem 1, such a mean is unique. In this connection we propose
the following

\bigskip

\begin{definition}
Let $K\colon I^{p}\rightarrow I$ be a continuous and monotone mean which is
invariant with respect to the mean type mapping $\mathbf{M}:=(M_{1},\dots
,M_{p})$.

For each set $S\subset \mathbb{N}_{p}:$

(i) the mean $K_S(\mathbf{M})$ is called a $K$-complementary averaging of
the means $\{M_{i}:i\in S\}$ with respect to the invariant mean-type mapping 
$\mathbf{M}=(M_{1},\dots ,M_{p})$;

(ii) the mean-type mapping $\mathbf{K}_S(\mathbf{M})$ given by (1) is called
a $K$-complementary averaging of the mean-type mapping $\mathbf{M}%
=(M_{1},\dots ,M_{p})$ with respect to the means $\{M_{i}:i\in S\}$, and the
set 
\begin{equation*}
\mathfrak{K}\left( K,\mathbf{M}\right) :=\left\{ \mathbf{K}_S(\mathbf{M}%
):S\subseteq \mathbb{N}_{p}, S \ne \emptyset\right\}
\end{equation*}%
is called the family of all $K$-complementary averaging of the mean-type
mapping $\mathbf{M}=(M_{1},\dots ,M_{p}).$
\end{definition}

We can now reapply this result to the complementary of the establish a $K$%
-complementary of $\mathbf{K}_S(\mathbf{M})$ for the set $\mathbb{N}%
_{p}\setminus S $. More precisely we obtain

\begin{corollary}
\label{cor:KS} Under the assumptions of Theorem~\ref{thm:KS} there exists
unique mean $K_S^*(\mathbf{M}) \colon I^p \to I$ such that $K$ is $\mathbf{K}%
_S^*(\mathbf{M})$-invariant, where $\mathbf{K}_S^*(\mathbf{M})\colon I^p \to
I^p$ is given by 
\begin{equation*}
[\mathbf{K}_S^*(\mathbf{M})]_i:=%
\begin{cases}
K_S(\mathbf{M}) & \text{ for }i \in S, \\ 
K_S^*(\mathbf{M}) & \text{ for }i \in \mathbb{N}_p \setminus S.%
\end{cases}%
\end{equation*}
Moreover $\min(M_i \colon i \in \mathbb{N}_p \setminus S) \le K_S^*(\mathbf{M%
}) \le \max(M_i \colon i \in \mathbb{N}_p \setminus S)$.
\end{corollary}

Let us underline that the value $K_S^*(\mathbf{M})$ does \textbf{not} depend
on $\mathbf{M}$ explicitly. The whole system of dependences is illustrated
on Figure~\ref{fig}.

\bigskip

Observe that, as the mean $K_S^*(\mathbf{M})$ is uniquely determined, we
obtain 
\begin{equation*}
\mathbf{K}_S^*(\mathbf{M})\in \mathfrak{K}\left( K,\mathbf{M}\right) \iff 
\text{ all means }(M_i \colon i \in \mathbb{N}_p \setminus S) \text{ are
equal to each other.}
\end{equation*}



\section{\label{sec:ASFE}\protect\bigskip\ Application in solving functional
equations}

\begin{theorem}
Let $\mathbf{M}=(M_{1},\dots ,M_{p})$ be a mean-type mapping such that $%
M_{1},\dots ,M_{p}:\left( 0,\infty \right) ^{p}\rightarrow \left( 0,\infty
\right) $ are strictly monotonic and homogeneous. Then

\ (i) the sequence $\left( \mathbf{M}^{n}:n\in \mathbb{N}\right) $ of
iterates of $\mathbf{M}$ converge uniformly on compact subsets to a
mean-type mapping $\mathbf{K}=\left( K,...,K\right) $, where $K$ is a unique 
$\mathbf{M}$-invariant mean.

(ii) $K$ is monotone, homogeneous and for every $S\subset \mathbb{N}_{p}$
the iterates of $\mathbf{K}_S(\mathbf{M})$ converge uniformly on compact
subsets to a mean-type mapping $\mathbf{K}=\left( K,...,K\right) $;

(iii) a function $F:\left( 0,\infty \right) ^{p}\rightarrow \mathbb{R}$ is
continuous on the diagonal $\Delta \left( \left( 0,\infty \right)
^{p}\right) :=\left\{ \left( x_{1},\dots ,x_{p}\right) \in I:x_{1}=\dots
=x_{p}\right\} $ and satisfies the functional equation%
\begin{equation}
F\circ \mathbf{M}=F  \tag{2}
\end{equation}%
if and only if $F=\varphi \circ K$, where $\varphi :\left( 0,\infty \right)
\rightarrow \mathbb{R}$ is an arbitrary continuous function of a single
variable;

(iv) a function $F:\left( 0,\infty\right) ^{p}\rightarrow \mathbb{R}$ is
continuous on the diagonal $\Delta \left( \left( 0,\infty \right)
^{p}\right) $ and satisfies the simultaneous system of functional equations%
\begin{equation}
F\circ \mathbf{K}_S(\mathbf{M})=F\text{, \ \ \ \ \ \ }S\subset \mathbb{N}_{p}%
\text{;}  \tag{3}
\end{equation}%
if and only if $F=\varphi \circ K$, where $\varphi :\left( 0,\infty \right)
\rightarrow \mathbb{R}$ is an arbitrary continuous function of a single
variable (so (2) and (3) are equivalent)
\end{theorem}

\begin{proof}
The homogeneity and monotonicity of $M_{1},\dots ,M_{p}$ imply their
continuity \cite[Theorem 2]{JM2013}), so the invariance principle implies
(i).


Now we prove that $K$ is monotone. Indeed, take two vectors $v,w \in
(0,\infty)^p$ such that $v_i\le w_i$ for all $i \in \mathbb{N}_p$ and $%
v_{i_0}<w_{i_0}$ for certain $i_0 \in \mathbb{N}_p$. Then, as each $M_i$ is
monotone, there exists a constant $\theta \in (0,1)$ such that $M_i(v) \le
\theta M_i(w)$ for all $i \in \mathbb{N}_p$.

Then for all $n \in\mathbb{N}$ and $i \in \mathbb{N}_p$ we have 
\begin{align*}
\big[\mathbf{M}^{n}(v)\big]_i &=\big[\mathbf{M}^{n-1}(M_1(v),\dots,M_p(v))%
\big]_i \le\big[\mathbf{M}^{n-1}(\theta M_1(w),\dots,\theta M_p(w))\big]_i \\
&=\theta \big[\mathbf{M}^{n-1}(M_1(w),\dots,M_p(w))\big]_i= \theta \big[%
\mathbf{M}^{n}(w)\big]_i.
\end{align*}

In a limit case as $n \to \infty$ in view of the first part of this
statement we obtain $K(v)\le \theta K(w)<K(w)$. Thus $K$ is monotone, which
is (ii).



(iii) \ Assume first that $F:\left( 0,\infty \right) ^{p}\rightarrow \mathbb{%
R}$ that is continuous on the diagonal $\Delta \left( \left( 0,\infty
\right) ^{p}\right) $ and satisfies equation (2). Hence, by induction, 
\begin{equation*}
F=F\circ \mathbf{M}^{n},\text{ \ \ \ \ }n\in \mathbb{N},
\end{equation*}%
By (ii) the sequence $\left( \mathbf{M}^{n}:n\in \mathbb{N}\right) $
converges to $\mathbf{K}=\left( K,...,K\right) $. Since $F$ is continuous on 
$\Delta \left( \left( 0,\infty \right) ^{p}\right) $, we hence get for all $%
\mathbf{x}\in \left( 0,\infty \right) ^{p}$, \ 
\begin{eqnarray*}
F\left( \mathbf{x}\right) &=&\lim_{n\rightarrow \infty }F\left( \mathbf{M}%
^{n}\left( \mathbf{x}\right) \right) =F\left( \lim_{n\rightarrow \infty }%
\mathbf{M}^{n}\left( \mathbf{x}\right) \right) =F\left( \mathbf{K}\left( 
\mathbf{x}\right) \right) \\
&=&F\left( K\left( \mathbf{x}\right) ,...,K\left( \mathbf{x}\right) \right) .
\end{eqnarray*}%
Setting%
\begin{equation*}
\varphi \left( t\right) :=F\left( t,...,t\right) \text{, \ \ \ \ \ }t\in
\left( 0,\infty \right) \,\text{,}
\end{equation*}%
we hence get $F\left( \mathbf{x}\right) =\varphi \left( K\left( \mathbf{x}%
\right) \right) $ for all $\mathbf{x}\in \left( 0,\infty \right) ^{p}$.

To prove the converse implication, take an arbitrary function $\varphi
:I\rightarrow \mathbb{R}$ and put $F:=\varphi \circ K$. Then, for all $%
\mathbf{x}\in \left( 0,\infty \right) ^{p},$ making use of the $K$%
-invariance with respect to $\mathbf{M}$, we have 
\begin{equation*}
F\left( \mathbf{M}\left( \mathbf{x}\right) \right) =\left( \varphi \circ
K\right) \left( \mathbf{M}\left( \mathbf{x}\right) \right) =\varphi \left(
K\left( \mathbf{M}\left( \mathbf{x}\right) \right) \right) =\varphi \left(
K\left( \mathbf{x}\right) \right) =F\left( \mathbf{x}\right) ,
\end{equation*}%
which completes the proof of (iii).

(iv) we omit similar argument.
\end{proof}

Part (ii) of this result gives rise to the following \ 

%
%
%
%
%
%
%

\subsection{General complementary process}

Once we have a mean-type $\mathbf{M} \colon I^p \to I^p$ and a continuous
and monotone mean $K \colon I^p \to I$ which is $\mathbf{M}$-invariant let $%
\mathfrak{K}^+\left( \mathbf{M},K\right)$ be the smallest family of
mean-type mappings containing $\mathbf{M}$ which is closed under $K$%
-complementary averaging.

More precisely for every $\mathbf{X} \in \mathfrak{K}^+\left( \mathbf{M}%
,K\right)$ and nonempty subset $S \subseteq \mathbb{N}_p$ we have $\mathbf{K}%
_S(\mathbf{X}) \in \mathfrak{K}^+\left( \mathbf{M},K\right)$, too. We also
define a family of means 
\begin{equation*}
\mathfrak{K}_0\left( \mathbf{M},K\right):=\big\{ \lbrack \mathbf{X}]_i
\colon \mathbf{X}\in \mathfrak{K}^+\left( \mathbf{M},K\right)\text{ and }i
\in \mathbb{N}_p\big\}
\end{equation*}

Obviously using notions from Theorem~\ref{thm:KS} and Corollary~\ref{cor:KS}
we have 
\begin{equation*}
\mathfrak{K}^{+}\left( \mathbf{M},K\right) \supseteq \mathfrak{K}^{+}\left( 
\mathbf{K}_{S}(\mathbf{M}),K\right) \supseteq \mathfrak{K}^{+}\left( \mathbf{%
K}_{S}^{\ast }(\mathbf{M}),K\right)
\end{equation*}%
Furthermore we have the following

\begin{proposition}
\label{prop:1} Given an interval $I \subset \mathbb{R}$, $p \in \mathbb{N}$,
and a mean-type mapping $\mathbf{M}:=(M_1,\dots,M_p) \colon I^p \to I^p$
which is invariant with respect to some continuous and monotone mean $K
\colon I^p \to I$. Then 
\begin{equation*}
\min(M_1,\dots,M_p) \le X \le \max(M_1,\dots,M_p) \qquad \text{ for all }X
\in \mathfrak{K}_0\left( \mathbf{M},K\right).
\end{equation*}
\end{proposition}

Its inductive proof is obvious in view of Theorem~\ref{thm:KS} (moreover
part).

Now we prove that complementary means preserve symmetry. 
%
%

\begin{proposition}
If a continuous and monotone mean $K\colon I^{p}\rightarrow I$ is invariant
with respect to a mean-type mapping $\mathbf{M}:=(M_{1},\dots ,M_{p})\colon
I^{p}\rightarrow I^{p}$ such that all $M_{i}$-s are symmetric, then $K$ and
all means in $\mathfrak{K}_{0}\left( \mathbf{M},K\right) $ are symmetric.
\end{proposition}

\begin{proof}
Fix a nonconstant vector $x\in I^{p}$ and a permutation $\sigma $ of $%
\mathbb{N}_{p}$. First observe that $K(x)=K\circ \mathbf{M}(x)=K\circ 
\mathbf{M}(x\circ \sigma )=K(x\circ \sigma )$, which implies that $K$ is
symmetric.

As the family $\mathfrak{K}_0\left( \mathbf{M},K\right)$ is generating by
complementing, we need to show that symmetry is preserved by a single
complement. Therefore it is sufficient to show that the mean $K_S(\mathbf{M}%
) $ defined in Theorem~\ref{thm:KS} is symmetric. However, using the notions
therein, we have 
\begin{equation*}
K \circ (\mathbf{K}_S(\mathbf{M}))(x)=K(x)=K(x\circ\sigma)=K \circ (\mathbf{K%
}_S(\mathbf{M}))(x\circ\sigma).
\end{equation*}
By monotonicity of $K$, if $K_S(\mathbf{M})(x)<K_S(\mathbf{M})(x\circ\sigma)$
we would have 
\begin{equation*}
K \circ (\mathbf{K}_S(\mathbf{M}))(x)=K(x)<K \circ (\mathbf{K}_S(\mathbf{M}%
))(x\circ\sigma)
\end{equation*}
contradicting the above equality. Similarly we exclude the case $K_S(\mathbf{%
M})(x)>K_S(\mathbf{M})(x\circ\sigma)$. Therefore $K_S(\mathbf{M})(x)=K_S(%
\mathbf{M})(x\circ\sigma)$ which, as $x$ and $\sigma$ were taken
arbitrarily, yields the symmetry of $K_S(\mathbf{M})$.
\end{proof}


\section{\label{sec:Beta} An applications to Beta-type means}

Following \cite{HimMat2018}, for a given $k\in \mathbb{N}$ we define a $p$%
-variable Beta-type mean $\mathcal{B}_{p}\colon \mathbb{R}%
_{+}^{k}\rightarrow \mathbb{R}_{+}$ by 
\begin{equation*}
\mathcal{B}_{p}(x_{1},\dots ,x_{p}):=\bigg(\frac{px_{1}\cdots x_{p}}{%
x_{1}+\dots +x_{p}}\bigg)^{\frac{1}{p-1}}.
\end{equation*}%
This is a particular case of so-called biplanar-combinatoric means (Media
biplana combinatoria) defined in Gini \cite{G} and Gini--Zappa \cite{GZ}. 


In order to formulate the next results, we adapt the notation that $A$, $G$
and $H$ are arithmetic, geometric and harmonic means of suitable dimension,
respectively.

\bigskip In \cite{JM20Coll}, the invariance $G\circ \left( A,H\right) =G,$
equivalent to the Pythagorean proportion, has been extended for arbitrary $%
p\geq 3$. In case $p=3$ it takes the form $G\circ \left( A,F,H\right) =G$,
where 
\begin{equation*}
F\left( x_{1},x_{2},x_{3}\right) =:\frac{x_{2}x_{3}+x_{3}x_{1}+x_{1}x_{2}}{%
x_{1}+x_{2}+x_{3}},\text{ \ \ \ \ }x_{1},x_{2},x_{3}>0,
\end{equation*}%
and $H\leq F\leq A$ $.$ Hence, making use of Corollary 1 with $p=3,$ $K=G$, $%
S=\left\{ 1,2\right\} $ we obtain the following

\begin{remark}
For all $x_{1},x_{2},x_{3},$ the following inequality holds%
\begin{equation*}
H\left( x_{1},x_{2},x_{3}\right) \leq \mathcal{B}_{3}(x_{1},x_{2},x_{p})\leq
A\left( x_{1},x_{2},x_{3}\right) ,
\end{equation*}%
and the inequalities are sharp for nonconstant vectors $x=\left(
x_{1},x_{2},x_{3}\right) \in \left( 0,\infty \right) ^{3}$.
\end{remark}

Passing to the main part of this section, first observe the following
easy-to-see lemma.

\begin{lemma}
\label{lem:1} Let $p\in \mathbb{N}$, $p\geq 2$. Then there exists exactly
one mean $M \colon I^p \to I$ such that $G\circ\big(A,\underbrace{M,\dots,M}%
_{(p-1)\text{ times}}\big)=G$. Furthermore $M=\mathcal{B}_{p}$. 
\end{lemma}

Its simple prove which is a straightforward implication of Theorem~\ref%
{thm:KS} is omitted. 
%
Based on this lemma it is natural to define a mean-type mapping $\mathbf{B}%
\colon I^p \to I^p$ by $[\mathbf{B}]_1:=A$ and $[\mathbf{B}]_i:=\mathcal{B}%
_p $ for all $i \in\{2,\dots,p\}$. Then we have $G \circ\mathbf{B}=G$, which
implies that the geometric mean is the unique $\mathbf{B}$-invariant mean.

We are now going to establish the set $\mathfrak{K}^{+}\left( \mathbf{B}%
,G\right) $. It is quite easy to observe that all means in $\mathfrak{K}%
_{0}\left( \mathbf{B},G\right) $ are of the form $\mathcal{H}_{p,\alpha }\colon
I^{p}\rightarrow I$ ($\alpha \in \mathbb{R}$) given by 
\begin{equation*}
\mathcal{H}_{p,\alpha }(x_{1},\dots ,x_{p}):=\big(x_{1}\cdots x_{p}\big)^{\frac{%
1-\alpha }{p}}\bigg(\frac{x_{1}+\dots +x_{p}}{p}\bigg)^{\alpha }
\end{equation*}%
including $\mathcal{B}_{p}=\mathcal{H}_{p,-\frac{1}{p-1}}$.
In the next lemma we show some elementary
properties of the family $(\mathcal{H}_{p,\alpha })$.

\begin{lemma}
Let $p \in \mathbb{N}$. Then

\begin{enumerate}[(i)]
\item \label{H.0} $\mathcal{H}_{p,\alpha}$ is reflexive for all $\alpha \in \mathbb{R}$%
, that is $\mathcal{H}_{p,\alpha}(x,\dots,x)=x$ for all $x \in \mathbb{R}_+$,\newline

\item \label{H.05} $\mathcal{H}_{p,\alpha}$ is continuous for all $\alpha \in \mathbb{R%
}$ (as a $p$-variable function), \newline

\item \label{H.1} $\mathcal{H}_{p,\alpha}$ is a strict mean for all $\alpha \in
[-\tfrac1{p-1},1]$,\newline

\item \label{H.2} $\mathcal{H}_{p,\alpha}$ is a symmetric function for all $\alpha \in 
\mathbb{R}$, that is $\mathcal{H}_{p,\alpha}(x \circ \sigma)=\mathcal{H}_{p,\alpha}(x)$ for all $%
x \in \mathbb{R}_+^p$ and a permutation $\sigma$ of $\mathbb{N}_p$,\newline

\item \label{H.3} $\mathcal{H}_{p,1}$ and $\mathcal{H}_{p,0}$ are $p$-variable arithmetic and
geometric means, respectively,\newline

\item \label{H.4} $\mathcal{H}_{p,\alpha}$ is increasing with respect to $\alpha$,
that is $\mathcal{H}_{p,\alpha}(x) < \mathcal{H}_{p,\beta}(x)$ for every nonconstant vector $x
\in \mathbb{R}_+^p$ and $\alpha,\beta \in\mathbb{R}$ with $\alpha<\beta$.
\end{enumerate}
\end{lemma}

\begin{proof}
By the definition of $\mathcal{H}_{p,\alpha}$ we can easily verify that \eqref{H.0}, %
\eqref{H.05}, \eqref{H.2} and \eqref{H.3} holds.

From now on fix a nonconstant vector $x=(x_1,\dots,x_p) \in \mathbb{R}_+^p$.
By \eqref{H.2} we may assume without loss of generality that $x_1\le\dots\le
x_p$. Denote briefly 
\begin{equation*}
g:=\sqrt[p]{x_1\cdots x_p} \qquad \text{ and }\qquad a:=\frac{x_1+\dots+x_p}%
p.
\end{equation*}
By Cauchy inequality we have $x_1<g<a<x_p$. Moreover by the definition $%
\mathcal{H}_{p,\alpha}(x)=a^\alpha g^{1-\alpha}$. Thus for all $\alpha<\beta$ we have 
\begin{equation*}
\mathcal{H}_{p,\beta}(x) =a^\beta g^{1-\beta} =a^\alpha g^{1-\alpha} \big(\tfrac ag %
\big)^{\beta-\alpha} > a^\alpha g^{1-\alpha} =\mathcal{H}_{p,\alpha}(x),
\end{equation*}
which completes the proof of \eqref{H.4}. The only remaining part to be
proved is \eqref{H.1}. However, applying \eqref{H.4}, it is sufficient to
show that 
\begin{equation*}
\mathcal{H}_{p,1}(x) < \max(x)=x_p\qquad\text{ and }\qquad \mathcal{H}_{p,-\frac1{p-1}}(x) >
\min(x)=x_1.
\end{equation*}
By \eqref{H.3} we immidiatelly obtain $\mathcal{H}_{p,1}(x)=a<x_p$. 
For the second part observe that 
\begin{align*}
\mathcal{H}_{p,-\frac1{p-1}}(x)=\sqrt[p-1]{\frac{g ^p}a}=\sqrt[p-1]{\frac{x_1\dots x_p}%
a}>\sqrt[p-1]{\frac{x_1\dots x_p}{x_p}}=\sqrt[p-1]{x_1\dots x_{p-1}}\ge x_1,
\end{align*}
which completes the proof.
\end{proof}

Now we generalize Lemma~\ref{lem:1} to the following form

\begin{lemma}
\label{lem:2} Let $p \in \mathbb{N}$, $p\ge2$ and $\alpha \in \mathbb{R}^p$.
Then $G \circ (\mathcal{H}_{p,\alpha_1},\dots,\mathcal{H}_{p,\alpha_p})=G$, if and only if $%
\alpha_1+\dots+\alpha_p=0$.
\end{lemma}

Its proof is obvious in view of the identity $G\circ (\mathcal{H}_{p,\alpha
_{1}},\dots ,\mathcal{H}_{p,\alpha _{p}})=\mathcal{H}_{p,\frac{1}{p}(\alpha _{1}+\dots +\alpha
_{p})}$. Having this proved, let us show the next important result.

\begin{theorem}
Let $p \ge 2$. Then 
\begin{align*}
&\mathfrak{K}^+\big((A,\underbrace{\mathcal{B}_p,\dots,\mathcal{B}_p}_{(p-1)%
\text{ times}}),G\big) \\
&\qquad\subseteq \big\{ (\mathcal{H}_{p,\alpha_1},\dots,\mathcal{H}_{p,\alpha_p}) \mid
\alpha_1,\dots,\alpha_p \in \mathbb{Q} \cap [-\tfrac1{p-1},1] \text{ and }%
\alpha_1+\dots+\alpha_p=0 \big\}.
\end{align*}
In particular 
\begin{equation*}
\mathfrak{K}_0=\big((A,\mathcal{B}_p,\dots,\mathcal{B}_p),G\big) \subseteq %
\big\{ \mathcal{H}_{p,\alpha} \mid \alpha \in \mathbb{Q} \cap [-\tfrac1{p-1},1] \big\}.
\end{equation*}
\end{theorem}

\begin{proof}
First observe that in view of Lemma~\ref{lem:1} $G$ is $\mathbf{B}$%
-invariant, and thus the set $\mathfrak{K}^+(\mathbf{B},G)$ is well-defined.
Now denote briefly 
\begin{equation*}
\Lambda:=\big\{ (\mathcal{H}_{p,\alpha_1},\dots,\mathcal{H}_{p,\alpha_p}) \mid
\alpha_1,\dots,\alpha_p \in \mathbb{Q} \cap [-\tfrac1{p-1},1] \text{ and }%
\alpha_1+\dots+\alpha_p=0 \big\}.
\end{equation*}

Obviously $\mathbf{B} \in \Lambda$, so it is sufficient to prove that $%
\Lambda$ is closed with respect to $G$-complementary averaging. To this end
take an arbitrary vector $(\alpha_1,\dots,\alpha_p)$ real numbers such that $%
\mathbf{H}:=(\mathcal{H}_{p,\alpha_1},\dots,\mathcal{H}_{p,\alpha_p})\in \Lambda$ and a nonempty
subset $S \subset \mathbb{N}_p$.

By Theorem~\ref{thm:KS} there exists exactly one mean $G_S(\mathbf{H})
\colon I^p \to I$ such that $G$ is $\mathbf{G}_S(\mathbf{H})$-invariant,
where $\mathbf{G}_S(\mathbf{H})\colon I^{p}\rightarrow I^{p}$ is given by 
\begin{equation*}
\big[ \mathbf{G}_S(\mathbf{H})\big]_{i}:=%
\begin{cases}
\mathcal{H}_{p,\alpha_i} & \text{ for }i\in \mathbb{N}_{p}\setminus S, \\ 
G_S(\mathbf{H}) & \text{ for }i\in S.%
\end{cases}%
\end{equation*}

On the other hand, in view of Lemma~\ref{lem:2} we obtain that $G$ is
invariant with respect to the mean-type mapping $\mathbf{H}_0 \colon I^p \to
I^p$ given by 
\begin{equation*}
\big[ \mathbf{H}_0\big]_{i}:=%
\begin{cases}
\mathcal{H}_{p,\alpha_i} & \text{ for }i\in \mathbb{N}_p\setminus S, \\ 
\mathcal{H}_{p,\beta} & \text{ for }i\in S,%
\end{cases}
\text{ where } \beta= \frac{1}{|S|} \sum_{i \in S} \alpha_i.
\end{equation*}%
%
%

As $G$ is both $\mathbf{G}_S(\mathbf{H})$-invariant and $\mathbf{H}_0$%
-invariant we obtain $G_S(\mathbf{H})=\mathcal{H}_{p,\beta}$, and consequently $%
\mathbf{G}_S(\mathbf{H})=\mathbf{H}_0$. Observe that $\mathbf{H}_0\in
\Lambda $ is a straightforward implication of the equality 
\begin{equation*}
\sum_{i \in \mathbb{N}_p\setminus S} \alpha_i+ |S| \beta=\sum_{i \in \mathbb{%
N}_p} \alpha_i=0.
\end{equation*}

Now we show that $\beta \in \mathbb{Q} \cap [-\tfrac1{p-1},1]$, which would
complete the proof. But this is simple in view of the definition of $\beta$
and the analogous property $\alpha_i \in \mathbb{Q} \cap [-\tfrac1{p-1},1]$,
which is valid for all $i \in S$. 
\end{proof}

\begin{remark}
Let us just put the reader attention that inclusions in the above theorem
are strict. More precisely, we can prove by simple induction that the
denominator of $\alpha $ in irreducible form has no prime divisors greater
than $p$.
\end{remark}


\bigskip

\begin{thebibliography}{9}
\bibitem{G} C. Gini, \textit{Di una formula comprensiva delle medie}. Metron 
\textbf{13} (1938), no. 2, 3--22.

\bibitem{GZ} C. Gini, G. Zappa, \textit{Sulle proprieta delle medie
potenziate e combinatorie}. Metron \textbf{13} (1938), no. 3, 21--31.

\bibitem{HimMat2018} M. Himmel, J. Matkowski, \textit{Beta-type means}. J.
Difference Equ. Appl. \textbf{24} (2018), no. 5, 753--772.

\bibitem{JM1999} J. Matkowski, Invariant and complementary quasi-arithmetic
means. Aequationes Math. \textbf{57} (1999), no. 1, 87--107.

\bibitem{JM2009} J. Matkowski, \textit{Iterations of the mean-type mappings. 
}Iteration theory (ECIT '08) , Grazer Math. Ber., \textbf{354}, Institut f%
\"{u}r Mathematik, Karl-Franzens-Universit\"{a}t Graz, 2009. 158--179.

\bibitem{JM2013} J. Matkowski, \textit{Iterations of the mean-type mappings
and uniqueness of invariant means}. Ann. Univ. Sci. Budapest. Sect. Comput. 
\textbf{41} (2013), 145--158.

\bibitem{JM20Coll} J. Matkowski, \textit{Means, generalized harmony
proportion and applications}. Colloq. Math. \textbf{160} (2020), no. 1,
109--118.

\bibitem{M-P} J. Matkowski, P. Pasteczka, \textit{Mean-type mappings and
invariance principle}. preprint, arXiv:2005.10623.
\end{thebibliography}
\end{document}